\documentclass[11pt]{amsart}
\usepackage{pdfsync}
\usepackage{fancyhdr}
\usepackage{float}
\usepackage[english]{babel}
\usepackage[utf8]{inputenc}
\usepackage[T1]{fontenc}
\usepackage{relsize,exscale}
\usepackage{setspace}
\usepackage{caption}
\usepackage{cite}     
\usepackage[pdf]{graphviz}
\usepackage{tikz}
\usetikzlibrary{cd}
\usepackage{pgfplots}
\pgfplotsset{compat=1.15}
\usepackage{mathrsfs}
\usetikzlibrary{arrows}
\usepackage{amsmath}
\usepackage{amssymb}
\usepackage{amsthm}
\usepackage{graphicx}
\usepackage[all,cmtip]{xy}

\newtheorem{thm}{Theorem}[section]
\newtheorem{lem}[thm]{Lemma}

\newtheorem{cor}[thm]{Corollary}

\newtheorem{prop}[thm]{Proposition}

\theoremstyle{definition}
\newtheorem{definition}[thm]{Definition}
\newtheorem{rmk}[thm]{Remark}

\newtheorem{ex}[thm]{Example}

\newtheorem{notation}[thm]{Notation}

\newtheorem{openquest}[thm]{Open question}

\usepackage{times}
\usepackage[]{hyperref}
\usepackage{hyperref}
\usepackage{xcolor}
\usepackage{multicol}
\hypersetup{
	colorlinks=true,     % Colorise les liens.
	urlcolor= black,     % Couleur des hyperliens.
	citecolor=red,        
	linkcolor=blue     % Couleur des liens internes.
	}

\begin{document}
\title{The rapid decay property for pairs of discrete groups}
\author{Indira Chatterji and Benjamin Zarka}

\newcommand{\etg}{\ell^2(G)}
\newcommand{\eog}{\ell^1(G)}
\newcommand{\cstr}{C^*_r(G)}
\newcommand{\etogh}{\ell^{(2,1)}(G,H)}
\newcommand{\cg}{\mathbb{C}G}
\newcommand{\ch}{\mathbb{C}H}
\newcommand{\cshyb}{B^*_r(G,H)}

\begin{abstract}
We generalize the notion of rapid decay property for a group $G$ to pairs of groups $(G,H)$ where $H$ is a finitely generated subgroup of $G$, where typically the subgroup $H$ does not have rapid decay. We deduce some isomorphisms in $K$-theory, and investigate relatively spectral injections in the reduced group $C^*$-algebra. Rapid decay property for the pair $(G,H)$ also gives a lower bound for the probability of return to $H$ of symmetric random walks on $G$.\end{abstract}
\maketitle
\section{Introduction}A \textit{length function} on a group $G$ is a map $\ell: G\rightarrow \mathbb{R}_+$  such that $\ell(e)=0$, where $e$ denotes the neutral element of $G$, $\ell(x)=\ell(x^{-1})$ and $\ell(xy)\leq\ell(x)+\ell(y)$ for all $x,y\in G$. We will use the \textit{word length} associated to a finite generating set of a finitely generated group $G$, defined by $\ell(e)=0$, and for all $x\neq e$ in $G$
\[
\ell(x)=\min\left\{n\in\mathbb{N},\ x=s_1...s_n|\ s_i\in S\cup S^{-1}\right\}.
\]
All word lengths on a finitely generated group $G$ are Lipschitz-equivalent and we will choose one that we denote $\ell$.
A finitely generated group $G$ has the {\it rapid decay property} if the set 
$${\bf{H}}_{\ell}^\infty(G)=\bigcap_{s\geq 0}\underbrace{\{\xi\in\etg\,|\,\xi(1+\ell)^{s}\in\etg\}}_{{\bf H}_{\ell}^s(G)}$$
embeds in the reduced C*-algebra $C^*_r(G)$, where $\ell$ is a length associated to some (or any) finite generating set of $G$. The sets ${\bf H}_{\ell}^s(G)$ are vector spaces and in case where the group has the rapid decay property, then for $s$ large enough are Banach algebras (see \cite{lafforgue2000proof}). In case where $G=\mathbb Z$, then the reduced C*-algebra is *-isometrically isomorphic to the continuous functions over the circle and ${\bf H}_{\ell}^\infty(\mathbb Z)$ are the smooth ones, which are in particular a subalgebra of the continuous functions and $\mathbb Z$ has the rapid decay property.

\noindent
This property traces back to Connes' book \cite{Connes}, with the free group case done by Haagerup in \cite{haagerup1978example}, and Jolissaint in \cite{joli} laying the ground for the theory. For $s$ large enough in a group $G$ with rapid decay, the space $\textbf{H}^s(G)$ is a Banach algebra and its embedding in $C^*_r(G)$ induces an isomorphism in $K$-theory (see \cite{valette}). This is the missing piece to use Lafforgue's work in \cite{lafforgue2002k} to establish the Baum-Connes conjecture for a large class of groups, including all lattices in a Lie group: the co-compact ones are conjectured to have the rapid decay property, but the non-cocompact ones are known not to, due to the presence of an exponentially distorted free abelian subgroup, which is an obstruction to the rapid decay property (see \cite{IC} for a general introduction to the rapid decy property). Another interest of the rapid decay property for a group $G$ is a lower bound for the return probability of the symmetric random walk on $G$, \cite{CPSC}. Even if a large class of groups has the rapid decay property (see \cite{centroid}), the existence either of abelian subgroups of arbitrarily large rank, or of an abelian subgroup with exponetial growth (for the length induced by the ambient group) are obstructions to the property, in fact the only ones so far, and the only one amongst 3-manifold groups, see \cite{CG}.

\noindent
Notice that even though $\eog$ is always a dense Banach sub-algebra of $\cstr$, the embedding is in general not spectral (see Section \ref{first}) but still induces an isomorphism at the level of K-theory when both the Bost and the Baum-Connes conjectures are known, like for the class of aTmenable groups, containing the one of amenable groups. We investigate the following Banach space norm, interpolating between $\etg$ and $\eog$.
\begin{definition}\label{def:(2,1)-norm}
    Given a countable group $G$ and a subgroup $H<G$, we define a {\it hybrid norm} as follows
$$\|f\|_{(2,1)}=\sqrt{\sum_{gH\in G/H} \|f|_{gH}\|_1^2}$$
and we denote by $\etogh$ the closure of $\cg$ with respect to this norm, where for $f\in\cg$, then $f|_{gH}$ is the restriction\footnote{Its $\ell^1$ norm will not depend on wheter one thinks of $f|_{gH}$ as an element of $\cg$ extended by 0 outside of the coset $gH$, or an an element of $\ch$.} of $f$ to the coset $gH\in G/H$. We shall say that the pair $(G,H)$ has the \textit{rapid decay property }, or that $G$ has {\it the hybrid rapid decay property with respect to $H$}
if there exists a positive constant $s$ such that the {\it hybrid Sobolev space of order $s$}
$${\bf{H}}_{\ell}^s(G,H)=\{\xi\in\etg\,|\,\xi(1+\ell)^{s}\in\etogh\}$$
embeds in $\cshyb$, the operator norm closure of $\cg$ acting on $\etogh$ (this is in general not a C*-algebra but a Banach *-algebra).
\end{definition}
When $H=\{e\}$, then the hybrid norm is the $\ell^2$-norm, and when $H=G$, it is the $\ell^1$ norm. If $H$ is a finite subgroup of $G$, then the hybrid norm is equivalent to the $\ell^2$-norm, and when $H$ is of finite index in $G$, it is equivalent to the $\ell^1$ norm, because on finite groups all those norms are equivalent\footnote{More generally, if $H$ and $K$ are two commensurable subgroups in $G$, then both hybrid norms are equivalent.} Hence, we can interpret the rapid decay property for the pair $(G,H)$ as an interpolation between the rapid decay property for $G$ (or for the pair $(G,\{e\})$, which may not hold) and for the pair $(G,G)$, which always holds since $\eog$ is a Banach algebra.
We show in Proposition \ref{thm:cas normal} that in case where $H<G$ is a normal subgroup, then the relative rapid decay property for the pair $(G,H)$ is equivalent to the rapid decay property for the quotient group. We investigate the relationship of the relative property with the rapid decay of the semi-regular representation as studied by Boyer \cite{boyer} (Definition \ref{def:RDrep}) and we obtain a generalization to non-normal subgroups of the stability of the rapid decay property under polynomial growth extensions (see Proposition \ref{prop:H poly}). The amenable obstruction to the rapid decay property gives a similar obstruction for pairs.
\begin{thm}\label{thm:cas coamenable}
Let $G$ be a finitely generated group and $H$ a
co-amenable subgroup of $G$.
Then, the pair $(G,H)$ has the rapid decay property if and only if the quotient graph $G/H$ has polynomial growth.
\end{thm}
We also establish the following, as a direct consequence of Proposition \ref{prop:H poly}.
\begin{cor}\label{thm:RD_H relhyp}
Let $G$ be a finitely generated group and $H$ a finitely generated subgroup. If $G$ is hyperbolic relative to $H$ and $H$ has polynomial volume growth, then the pair $(G,H)$ has the rapid decay property.
\end{cor}
An earlier version missed the polynomial growth assumption on the subgroup $H$, and it is sharp as showed by Jvbin Yao in \cite{JYao}.  The pair $(F_3,F_2)$ for instance doesn't have rapid decay, see Example \ref{free}.
We investigate the consequences of the rapid decay property for a pair of groups $(G,H)$ on K-theory.
\begin{thm}\label{prop:isoKtheory}
Let $G$ be a finitely generated group and $H$ be a subgroup of $G$, such that the pair $(G,H)$ has the rapid decay property. Then, there is $s$ large enough such that the space $\textbf{H}_\ell^s(G,H)$ is a Banach algebra and the natural inclusion in $\cshyb$ induces an isomorphism in K-theory, for any word length $\ell$ associated to a finite generating set of $G$. Moreover, when the subgroup $H$ has subexponential growth with respect to the induced length, this isomorphism extends to an isomorphism with the K-theory of $\cstr$ (Corollary \ref{cor: cas spectral dans C*}).
\end{thm}
%As a consequence, we obtain new unconditional completions of $\cg$ that have the same K-theory as $\cstr$ for a class of groups that doesn't have the rapid decay property, see Example \ref{Grigo}.

\smallskip

Similarly to the rapid decay property, for a pair of groups $(G,H)$, the relative rapid decay property also gives information on random walks, but over a Schreier graph of the left translation of $G$ on $G/H$.
\begin{thm}\label{thm:borne inf pour RW}
Let $G$ a finitely generated group, $H$ a subgroup of $G$ such that the pair $(G,H)$ has rapid decay. Then, there is a constant $C$ such that for any finitely supported symmetric measure $\mu$ on $G$, whose support generates $G$, we have, for all $n\in\mathbb{N}^*$ 
\[
n^{-2d}\leq C \left(\rho_{\mu}^{G/H}\right)^{-2n}P_{2n}(H,H)
\]
where $d$ is the exponent of the rapid decay property for the pair $(G,H)$, $P_{2n}(H,H)$ the probability of the return to the coset $H$ after $2n$ steps and $\rho_{\mu}^{G/H}$ the spectral radius of the random walk on $G/H$ (Definition \ref{def:spectralRad}).
\end{thm}
The text is organized as follows. Section \ref{first} explores the rapid decay property for pairs of groups, comparing it with the quasi-regular representation, and gives equivalent definitions for the rapid decay property for pairs. We shall see the first examples and the case where $H$ is a normal subgroup there (Proposition \ref{thm:cas normal}). Section \ref{relative} is devoted to the proofs of Theorems \ref{thm:cas coamenable} and \ref{thm:RD_H relhyp}. Section \ref{consequences} discusses the consequences of the rapid decay property for pairs in K-theory and finishes the proofs of Theorems \ref{prop:isoKtheory} and \ref{thm:borne inf pour RW}.
\subsection*{Acknowledgements} The material in this paper is part of the second author's PhD dissertation, and both authors thank Christophe Pittet and Hervé Oyono-Oyono for a careful reading of the material, as well as Fran\c{c}ois Gautero, Adrien Boyer, Maria Paula Gomez-Aparicio and Goulnara Arzhantseva for interesting conversations. We are also very grateful to Jvbin Yao for pointing out a mistake in an earlier version of the paper and for pointing out the argument in \cite{JYao}.
%%%%%%%%%%%%%%%%%%%%%%%%%%%%%%%%
\section{Basic definitions and first observations}\label{first}
%%%%%%%%%%%%%%%%%%%%%%%%%%%%
For a discrete group $G$, the {\it convolution product} of $f,\phi\in\cg$ is given by
\[f*\phi(x)=\sum_{z\in G} f(z)\phi(z^{-1}x)\]
for any $x\in G$. This corresponds to the algebra product on $\cg$ when one identifies an element $\gamma\in G$ with $\delta_\gamma$ the characteristic function on the singleton $\{\gamma\}$. This product extends to a Banach algebra product on $\eog$ and to a C*-algebra product on $\cstr$, but in general not to $\etg$ or $\etogh$ (see Definition \ref{def:(2,1)-norm}) unless $G$ is finite or the subgroup $H$ has finite index in $G$. 
\begin{definition}\label{def:hybrid} Given a finitely generated group $G$ and a subgroup $H$ in $G$, the \textit{hybrid operator norm} is defined by
\[
\|f\|_{h}:=\sup_{\|\phi\|_{(2,1)}=1} \|f*\phi\|_{(2,1)}.
\]
The closure of $\cg$ with respect to this operator norm defines the Banach algebra $\cshyb$ of Definition \ref{def:(2,1)-norm}. 
\end{definition}

\begin{rmk}
\label{rmk:injection l^1}
The Banach algebra $\cshyb$ naturally contains $\eog$. Indeed, for all
$f\in\eog$ and $h\in\etogh$, since the left translation of $G$ is an isometry on $\etogh$ we have 
\[
\|f*h\|_{(2,1)}=\|\sum_{y\in G} f(y)\delta_y*h\|_{(2,1)}\leq \sum_{y\in G} |f(y)|\ 
\|\delta_y*h\|_{(2,1)}=\|f\|_1\|h\|_{(2,1)}
\]
which implies that $\|f\|_{h}\leq\|f\|_1.$ It is a basic computation that we always have the following inequalities:
 \[
\|f \|_{2} \leq \|f \|_{(2,1)}\leq \|f \|_{1}
\]
which implies that $\eog\subseteq\etogh\subseteq\etg$, and those inclusions are strict as soon as $H$ is infinite and of infinite index in $G$. More generally, if $\Gamma<\Lambda<G$, then the hybrid norm with respect to $\Gamma$ is bounded above by the one with respect to $\Lambda$, but this is unclear in case of the hybrid operator norms in general.
\end{rmk}
\begin{definition}\label{def:SobolevNorm}
 Given a finitely generated group $G$ and a subgroup $H$ in $G$. Let $s\geq 0$, the hybrid Sobolev space of order $s$ of Definition \ref{def:(2,1)-norm} gives the space of \textit{hybrid rapidly decreasing function} as the following Fr\'echet space
\[
\textbf{H}_{\ell}^{\infty}(G,H)=\bigcap_{s\geq 0}\textbf{H}_{\ell}^s(G,H).
\] 
where $\ell$ is a length on $G$ given by any finite generating set and the norm closure of $\cg$ with respect to the norm given by $\|f\|_{s,(2,1)}=\|f(1+\ell)^s\|_{(2,1)}$. We denote by $\|f\|_{s,2}=\|f(1+\ell)^s\|_{2}$ the usual Sobolev norm studied by Jolissaint in \cite{joli}.
\end{definition}
\begin{rmk}\label{rmk:equiv(2,1)} 
Suppose that $H<G$ is any subgroup, and denote by $\pi:G\rightarrow G/H$ the quotient map, which induces a linear map $\pi_\sharp:\mathbb{C}G\rightarrow\mathbb{C}G/H$, defined by
 \[
\pi_\sharp(f)=\pi_\sharp\left(\sum_{\gamma\in G} f_\gamma \delta_\gamma\right)=\sum_{\gamma\in G} f_\gamma \delta_{\pi(\gamma)}=\sum_{q\in Q}\left(\sum_{\gamma\in\pi^{-1}(q)}f(\gamma)\right) \delta_{q},
  \]
and it is a direct computation that $\|f\|_{(2,1)}=\|\pi_\sharp(|f|)\|_2$, where the last norm is in ${\ell^2(G/H)}$.\end{rmk}
\begin{definition}\label{def:growth} Let $X$ be a locally finite graph, endowed with the graph metric $d$, that is the combinatorial length of a shortest path between two points. Given $R\geq 0$, the ball of radius $R$ with center $x\in X$ is $B(x,R)=\{y\in X|\,d(x,y)\leq R\}$ and we denote by $\gamma_{X}(x,R):=|B(x,R)|$ the associated growth function for $x\in X$. The growth is not necessarily homogeneous, and we say that the growth of $X$ is 
\begin{enumerate}
\item \textit{Polynomial} if there exists a polynomial $Q$ such that for all $x\in X$ we have 
\[\gamma_{X}(x,R)\leq Q(R).\]
\item \textit{Exponential} if there are $a\geq 1$, $b\geq 0$ and $x\in X$ such that
\[
\gamma_X(x,R)\geq e^{(aR+b)}.
\]
\item \textit{Intermediate} if it is neither polynomial nor exponential.
\item \textit{Subexponential} if there exists $x\in X$ such that  \[
\limsup_{n\to\infty} \left(\gamma_X(x,n)\right)^{\frac{1}{n}}=1.
 \] 
\end{enumerate}
\end{definition}
In case of a finitely generated group, its growth is defined as the growth of any Cayley graph obtained with a finite generating set, and doesn't depend of that choice.
\begin{rmk}\label{rmk : pol growth} 
If $G=\left<S\right>$ has polynomial volume growth, then there are constants $C,D\geq 1$ such that $|B(R)|\leq CR^D$ for any $R>0$ where $B(R)=\{x\in G \ | \ell(x)\leq R\}$ is the ball of radius $R$ in $G$. Hence, for any $f\in\cg$ supported on a ball of radius $R$
$$\|f\|_1=\sum_{x\in B(R)}|f(x)|\leq\sqrt{C(R+1)^D}\sqrt{\sum_{x\in B(R)}|f(x)|^2}=\sqrt{C}(R+1)^{D/2}\|f\|_2.$$
As a consequence, if $H<G$ is any subgroup, it will have polynomial growth as well, with respect to the induced length from $G$, and so there is a polynomial $P$ such that
$$\|f\|_{(2,1)}\leq P(R)\|f\|_{2}$$
for any $f\in \mathbb{C}G$ with support in $B(R)$.
\end{rmk}

\subsection{Equivalent definitions of the rapid decay property for a pair of groups}
The goal of this section is to prove the following relative version of Proposition 4.1 in \cite{IC}. We made the choice to consider left cosets in Definition \ref{def:(2,1)-norm}, we could have taken right ones and both theories are equivalent. We denote by $\etogh_+$ and $\mathbb{R}_+G$ the subset of positive-valued functions of  $\etogh$ and $\cg$.
\begin{prop}
\label{prop:equivalence rd_H}
Let $G$ be a finitely generated group and $H$ a subgroup of $G$. The following conditions are equivalent:
\begin{enumerate}
\item The pair $(G,H)$ has the rapid decay property (Definition \ref{def:(2,1)-norm}), namely there exists positive consant $s_0$ such that, for any $s\geq s_0$
$${\bf H}^s_{\ell}(G,H)\subseteq\cshyb.$$
\item There exists constants $C,D\geq 0$ such that for all $f\in\cg$ with support in $B(R)$ and $\phi\in\etogh$ we have (see Definition \ref{def:(2,1)-norm} for the norm)
\[
\|f*\phi\|_{(2,1)}\leq C(R+1)^D\|f\|_{(2,1)}\|\phi\|_{(2,1)}.
\]
\item There exists constants $C,D\geq 0$ such that for all $f\in\cg$ with support in $B(R)$, then (see Definition \ref{def:hybrid} for the norm)
\[ \|f\|_{h}\leq C(R+1)^D\|f\|_{(2,1)}.\]
%\item The pair $(G,H)$ has the rapid decay property with respect to positive-valued functions.
\item There exists constants $C,D\geq 0$ such that for all $\phi,\psi\in\etogh_+$, and $f\in\mathbb{R}_+G$ with support in $B(R)$ the following inequality holds:
\[
\sum_{h\in H} (f*\phi)^{\star} *\psi^{\star}(h)\leq C(R+1)^D\ \|f\|_{(2,1)}\,\|\phi\|_{(2,1)}\,\|\psi^{\star}\|_{(2,1)}
\]
where the exponent $\star$ denote the involution $f^{\star}(x)=f(x^{-1})$.
\end{enumerate}
\end{prop}
\begin{rmk}\label{rmk:positive}
Any function $f\in\cg$ can be written as $f=f_1-f_2+i(f_3-f_4)$ with  $f_k\in \mathbb{R}_+G$ for $k\in\{1,...,4\}$, and where for $k=1,3$, $f_k$ and $f_{k+1}$ have disjoint supports. We can hence rewrite the $(2,1)$-norm as follows:
\[
\|f\|_{(2,1)}=\sqrt{\sum_{gH\in G/H}\left(\sum_{x\in gH}\sqrt{\sum_{k=1}^4 f_k(x)^2}\right)^2}.
\]
Applying Cauchy-Schwarz inequality twice, we find constants $C, D\geq 0$ such that, for all $\phi\in \etogh_+$ 
and $f\in\mathbb{C}G$ with support in $B(R)$ the following inequality holds:
$$
\|f*\phi\|_{(2,1)}\leq \sum_{k=1}^4 \|f_k*\phi
\|_{(2,1)}\leq 4CR^D\|f\|_{(2,1)}\|\phi\|_{(2,1)}.$$
Writing $\phi=\phi_1-\phi_2+i(\phi_3-\phi_4)$ for $\phi\in \etogh$, we obtain that it is enough to prove (2) and (3) for positive-valued functions to deduce it for all the functions in $\cg$.
\end{rmk}
\begin{proof}[Proof of Proposition \ref{prop:equivalence rd_H}] The equivalence beetween (2) and (3) follows from the definition of operator norm. We will show that $(3)\Leftrightarrow (1)$, and then $(2)\Leftrightarrow (4)$.

\smallskip
\noindent
That (3) implies (1) follows from the spherical decomposition
$$f=\sum_{n\in{\bf N}}f_n$$ 
where $f_n$ is the restriction of $f$ over the sphere of radius $n$ given by $S_n:=\{x\in G\ | \ \ell(x)=n \}$. Indeed, suppose that $(3)$ holds, then 
there exists positive constants $C$ and $D$ such that (using Cauchy-Schwarz inequality):
\[
\|f\|_{h}\leq\sum_{n=0}^{\infty} \|f_n\|_{h}\leq C\sum_{n=0}^{\infty}(n+1)^D\|f_n\|_{(2,1)}\leq K\|f\|_{(2,1)}.
\]
for any $s\geq D+1$, with $K=C\frac{\pi}{\sqrt{6}}$ and where the norm is the one from Definition \ref{def:SobolevNorm}. Hence,  ${\bf H}_\ell^s(G,H)\subseteq \cshyb$ for any $s\geq D+1$.

The converse follows from $\|f\|_{s,(2,1)}\leq C(R+1)^s\|f\|_{(2,1)}$, for any $s\geq 0$ and for all $f\in \cg$ with support in $B(R)$.

\noindent For the equivalence beetween $(2)$ and $(4)$, according to Remark \ref{rmk:positive}, without loss of generality we can assume that all the functions are positive, so let $f\in\mathbb{R}_+G$ with support in $B(R)$, and $\phi,\psi\in \etogh_+$. By  decomposing the sum over $G$ in the convolution on right cosets, we have the following:
\[
\sum_{h\in H} (f*\phi)^**\psi^*(h)=
\sum_{Hg\in H\setminus G}\sum_{z\in Hg} (f*\phi)^*(z)\sum_{h\in H}\psi^*(z^{-1}h).
\]
Using the change of variable $y=z^{-1}h$, we get that for all $z\in Hg$, then $y^{-1}$ is an element of $Hg$. Moreover, using Cauchy-Schwarz inequality we obtain that
\[
\sum_{Hg\in H\setminus G}\sum_{z\in Hg} (f*\phi)^*(z)\sum_{y^{-1}\in Hg}\psi(y^{-1})\leq 
\|f*\phi\|_{(2,1)}\|\psi^*\|_{(2,1)}.
\]
Suppose $(2)$, then  there exists positive constants $C$ and $D$ such that
\[
\sum_{h\in H} (f*\phi)^**\psi^*(h)
\leq \|f*\phi\|_{(2,1)}\|\psi^*\|_{(2,1)}\leq  C(R+1)^D\ \|f\|_{(2,1)}\,\|\phi\|_{(2,1)}\,\|\psi^*\|_{(2,1)}.
\] and $(4)$ follows. Reciprocally, assume $(4)$, and choose $\psi$ as follows: 
$\psi=(f*\phi)^*$. By construction, $\|\psi^*\|_{(2,1)}=\|f*\phi\|_{(2,1)}$ and  by changing the variable $y=z^{-1}h$ we have
\[
\sum_{h\in H} (f*\phi)^**\psi^*(h)=
\sum_{h\in H} (f*\phi)^**(f*\phi)(h) 
\]
\[= \sum_{z\in G} (f*\phi)^*(z)\sum_{h\in H} f*\phi(z^{-1}h)
=
\sum_{Hg\in H\setminus G}\sum_{z\in Hg} f*\phi(z^{-1})\sum_{y\in z^{-1}H} f*\phi(y)
\]
for all $\phi\in \etogh_+$, and
 $f\in \mathbb{R}_+G$ with support in $B(R)$.
Since for $z\in Hg$, $yH$ and $z^{-1}H$ are the same left coset. Hence, we get that the last expression is equal to $\|f*\phi\|_{(2,1)}^2$.  With this choice of $\psi$, we obtain (using $(4)$) that 
\[
\|f*\phi\|_{(2,1)}^2=
\sum_{h\in H} (f*\phi)^**\psi^*(h)\leq C(R+1)^D\ \|f\|_{(2,1)}\|\phi\|_{(2,1)}\|f*\phi\|_{(2,1)}.
\] which concludes the equivalence.
\end{proof}
The following is a straightforward observation.
\begin{lem}[Subgroup stability]
\label{lem: stabilité sous-groupe} Let $(G,H)$ be a pair of groups with the rapid decay property. Then for any finitely generated group $K$ inbetween $H<K<G$, the pair $(K,H)$ has  the rapid decay property.
\end{lem}
\begin{proof}
For $f\in {\mathbb R}K$, the hybrid operator norm for the pair $(K,H)$ is bounded above by the one for the pair $(G,H)$, and we conclude because the hybrid norms coinicide.\end{proof}

In general, the relatively rapid decaying functions ${\bf H}_\ell^s(G,H)$ (see Definition \ref{def:(2,1)-norm}) do not form a Banach algebra, but as in the rapid decay case, we have the following.
\begin{prop}
\label{prop: Bs est une Banach alg} Let $G$ be a finitely generated group and $H$ a subgroup of $G$. If the pair $(G,H)$ has the rapid decay property, then there is $s_0$ large enough so that the space ${\bf H}_\ell^s(G,H)$ is a Banach algebra for every $s\geq s_0$ and every word length $\ell$ from a finite generating set of $G$.
\end{prop}
\begin{proof}
We follow Lafforgue's proof described in Proposition 8.15 in \cite{valette}, for $s\geq 1$ and $f,\psi\in \mathbb{C}G$, $y\in G$ we have that
%\[|f*\psi(y)|(1+\ell(y))^s\leq \sum_{z\in G} |f(z)\psi(z^{-1}y)(1+\ell(y))^s|\leq \sum_{z\in G} |f(z)\psi(z^{-1}y)(1+\ell(y)+l(z^{-1}y))^s|.\]
%Since  for all $a,b>0$ and $n\geq 1$ we have $(a+b)^n\leq 2^{n-1}(a^n+b^n)$, we can upper-bound the last summation by  
%\[2^{s-1} \sum_{z\in G} |f(z)||\psi(z^{-1}y)|(1+\ell(z))^s + 2^{s-1} \sum_{z\in G} |f(z)||\psi(z^{-1}y)|(1+\ell(z^{-1}y))^s.\]
%So we get that 
\[
(1+\ell(y))^s|(f*\psi)(y)|\leq
2^{s}\left(|f|(1+\ell)^s*|\psi|\right)(y)+2^s(|f|*|\psi|(1+\ell)^s)(y)                                        . \]
For a coset $gH\in G/H$, we compute the $\ell^1$-norm of the restriction of $f*\psi$ to the coset $gH$, which gives
$$
    \|(1+\ell)^s(f*\psi)|_{gH}\|_{1}\leq 2^{s}\left(\|(|f|(1+\ell)^s*|\psi|)|_{gH}\|_{1}+\|(|f|*|\psi|(1+\ell)^s)|_{gH}\|_{1}\right)                                        .$$
%Using that $2ab\leq a^2+b^2$ we obtain 
%\[|| (1+\ell)^s f*\psi ||_{l^1(gH)}^2\leq  2^{2s-1}\left(||(1+\ell)^s|f|*|\psi|||_{l^1(gH)}^2+|||f|*|\psi|(1+\ell)^s||_{l^1(gH)}^2\right). \]
Squaring and summing over all cosets $gH\in G/H$ we get that (see Definition \ref{def:SobolevNorm})
\begin{eqnarray*}
    \|f*\psi\|_{s,(2,1)}^2&\leq & 2^{2s+1}\left(\|\,|f|(1+\ell)^s*|\psi|\,\|_{(2,1)}^2+
\|\,|f|*|\psi|(1+\ell)^s\|_{(2,1)}^2\right)\\
&\leq &2^{2s+1}\left(\|\,|f|(1+\ell)^s\|_{(2,1)}^2\|\,|\psi|\,\|_h^2+
\|\,|f|\,\|_h^2\|\,|\psi|(1+\ell)^s\|_{(2,1)}^2\right)
\end{eqnarray*}
Since the pair $(G,H)$ has the rapid decay property, there exists a constant $C>0$ bounding the hybrid operator norm in terms of the weighted $(2,1)$ norm so that
 %\[||f*\psi||_{B^{s,2}_1}^2\leq  C^22^{2s-1}\left(||f||_{B^{s,2}_1}^2||\psi||_{B^{s',2}_1}^2+||f||_{B^{s',2}_1}^2||\psi||_{B^{s,2}_1}^2\right). \]Using the fact that for $s\leq s'$ we have $B^{s',2}_1G\subseteq {\bf H}^s(G,H)$, we can choose $s$ large enough ($s\geq\max\{s,s'\}$) to get that 
\[
\|f*\psi\|_{s,(2,1)}^2 \leq C^2 2^{2s+2}\|f\|_{s,(2,1)}^2 \|\psi\|_{s,(2,1)}^2.
 \]
 Hence, up to rescaling the norm, ${\bf H}_{\ell}^s(G,H)$ is a Banach algebra. 
\end{proof}

%%%%%
\subsection{Quasi-regular representations and rapid decay property for pairs} Recall that a unitary representation of a discrete countable group $G$ is a pair $(\pi,\mathcal{H})$ consisting of a group homomorphism $\pi:G\to\mathcal{U}(\mathcal{H})$ into the group of unitary operators of a Hilbert space $\mathcal{H}$. Any such representation extends by linearity to a representation of the group ring $\cg$ into $\mathcal{B}(\mathcal{H})$, the bounded operators on $\mathcal{H}$. The \textit{reduced $C^*$-algebra induced by $\pi$}, denoted by $C^*_\pi(G)$, is the closure of $\pi(\cg)$ in the operator norm of $\mathcal{B}(\mathcal{H})$, denoted by $\|\ \|_*$.
\begin{definition}[Boyer \cite{boyer}]\label{def:RDrep}
 Given any unitary representation $(\pi,\mathcal{H})$, of a finitely generated group $G$. We say that the representation has the {\it rapid decay property} if there is $s_0$ large enough and a constant $C>0$ such that for any $f\in\cg$ and any $s\geq s_0$, then 
$$\|\pi(f)\|_{*}\leq C\|f(1+\ell)^s\|_2.$$
\end{definition}
The rapid decay property for the group $G$ is equivalent to the rapid decay property for the regular representation. 
\begin{rmk}
   Let $(\pi_1,\mathcal{H}_1)$ and %
$(\pi_2,\mathcal{H}_2)$ two unitary representations of a discrete group $G$. 
According to Theorem F.4.4 in \cite{bekka2008kazhdan}, the representation $\pi_1$ is \textit{weakly contained} in $\pi_2$, and we denote $\pi_1\prec \pi_2$, if
$$\|\pi_1(f)\|_{\mathcal{B}(\mathcal{H}_1)}\leq\|\pi_2(f)\|_{\mathcal{B}(\mathcal{H}_2)}$$
for all $f\in\eog$. Hence, the rapid decay property for a representation implies it for all the representations weakly contained in it.     
\end{rmk}
Recall that for a subgroup $H$ of $G$, the left \textit{quasi-regular representation} is the representation induced from the left action of $G$ on the cosets $G/H$.
%, namely the homomorphism $\lambda_{G/H}: G \rightarrow \mathcal{U}(\ell^2(G/H))$ where
%
% \[\lambda_{G/H}(\gamma)(\xi)(gH)=
% \xi(\gamma^{-1}gH),\]
 % for all $\gamma\in G$, $gH\in G/H$, and $\tilde{\xi}\in\ell^2(G/H)$. 
 Extending this map by linearity to $\cg$, we obtain for $f\in\cg$ an operator $\lambda_{G/H}(f):\ell^2(G/H)\rightarrow \ell^2(G/H)$ where
\[
\lambda_{G/H}(f)(\tilde{\xi})(gH)=\sum_{z\in G} f(z)\tilde{\xi}(z^{-1}gH),
\]
for all $gH\in G/H$ and $\tilde{\xi}\in\ell^2(G/H)$. 
We will now see how those norms compare, with the following general result
%Notice that by construction the norm $||\lambda_{G/H}(f)||_{\mathcal{B}(l^2(G/H))}$ is a $C^*$-semi-norm where the adjoint operator is given by $(\lambda_{G/H}(f))^*=\lambda_{G/H}(f^*)$.
%
\begin{lem}
\label{lem: inclusion dans la semi-regul}\label{lem:operateur quotient vs hybride} Let $G$ be a finitely generated group, and $H<G$ be a subgroup.
\begin{enumerate}
    \item For any $f\in\cg$, then $\|\lambda_{G/H}(f)\|_*\leq \|f\|_{h}$ and the algebra $B^*_r(G,H)$ embeds in $C^*_{\lambda_{G/H}}(G)$, the $C^*$-algebra induced by the quasi-regular representation.
    \item For all $f\in\mathbb{R}_+G$, then $\|f\|_*\leq\|\lambda_{G/H}(f)\|_{*}=\|f\|_{h}.$
\end{enumerate}
\end{lem}
\begin{proof}
(1) 
We choose a system of representatives $R\subseteq G$ for $G/H$, and for any $\eta\in\ell^2(G/H)$ with norm one, we define $\tilde{\eta}_R$ to be the function on $G$ whose support intersects each coset $gH$ in the point $g\in R$. For any such $g\in R$ we have that $\tilde{\eta}_R(g)=\eta(gH)$, so that $\|\tilde{\eta}_R\|_{(2,1)}=\|\eta\|_{2}=1$. Hence
    \begin{eqnarray*}
\|\lambda_{G/H}(f)\eta\|_2^2&=& 
\sum_{gH\in G/H}\left|\sum_{z\in G}f(z)\eta(z^{-1}gH)\right|^2\\
&=&\sum_{gH\in G/H} \left|\sum_{z\in G} f(z)\sum_{x\in gH}\tilde{\eta}_R(z^{-1}x)\right|^2\\
&\leq& \sum_{gH\in G/H}\left(\sum_{x\in gH} \left|\sum_{z\in G} f(z)\tilde{\eta}_R(z^{-1}x)\right |\right)^2\\
&=&\|f*\tilde{\eta}_R\|_{(2,1)}^2\leq\|f\|_h^2.
\end{eqnarray*}
We conclude by taking the suppremum over all $\eta\in\ell^2(G/H)$ of norm one.

(2)
Let $\xi\in\etg$ of norm 1. Define the function $\tilde{\xi}\in\ell^2(G/H)$ by assigning to each coset $gH\in G/H$, the $\ell^2$ norm of $\xi$ restricted to that coset, namely $\tilde{\xi}(gH)=\|\xi|_{gH}\|_{2}$, so that $\tilde{\xi}$ has norm 1 in $\ell^2(G/H)$. For $f\in\mathbb{R}_+G$ we have that
\begin{eqnarray*}\|f*\xi\|^2_{2}&=&\sum_{gH\in G/H}\|(f*\xi)|_{gH}\|_{2}^2\leq\sum_{gH\in G/H}\left(\sum_{z\in G} f(z)\underbrace{\|(\delta_z*\xi)|_{gH}\|_2}_{\tilde{\xi}(z^{-1}gH)}\right)^2\\ &=&\|\lambda_{G/H}(f)(\tilde{\xi})\|^2_{2}\leq\|\lambda_{G/H}(f)\|^2_{*}
\end{eqnarray*}
and we conclude that the inequality in (2) holds by taking the supremum over $\xi\in\etg$ of norm 1.
Let us now turn to the equality, according to part (1) it remains to show that $\|f\|_{h}\leq\|\lambda_{G/H}(f)\|_{*}$ for all $f\in \mathbb{R}_+G$.
For any $\phi\in\etogh$ of norm one, we define $\tilde{\phi}\in\ell^2(G/H)$ by $\tilde{\phi}(gH)=\|\phi|_{gH}\|_{(2,1)}(=\|\phi|_{gH}\|_1)$ for any $gH\in G/H$, so that $\|\tilde{\phi}\|_{2}=\|\phi\|_{(2,1)}=1$ and 
 \begin{eqnarray*}\|f*\phi\|_{(2,1)}
 %&=&\sqrt{\sum_{gH\in G/H}\|(f*\phi)|_{gH}\|_{1}^2}= \sqrt{\sum_{gH\in G/H}\left(\sum_{z\in G}f(z)\tilde{\phi}(z^{-1}gH)\right)^2}\\
 &\leq&\|\lambda_{G/H}(f)(\tilde{\phi})\|_{2}\leq \|\lambda_{G/H}(f)\|_{*}\end{eqnarray*}
 so that we conclude by taking the supremum over $\phi\in\etogh$ of norm one.
\end{proof}

\begin{lem}
\label{lem: Bs s'injecte dans C*}
Let $G$ be a finitely generated group. If there is a subgroup $H$ such that the pair $(G,H)$ has the rapid decay property, then there exists $s_0\geq 0$ such that ${\bf H}^s_{\ell}(G,H)$ embeds in $C^*_r(G)$ for every $s\geq s_0$.
\end{lem}
\begin{proof}
According to Lemma \ref{lem:operateur quotient vs hybride},
for all $f\in\mathbb{R}_+G$, we have $\|f\|_*\leq \|f\|_{h}$. Since the pais $(G, H)$ has the rapid decay property, by using the decomposition of Remark \ref{rmk:positive}, we get constants $K,s>0$ such that for all $f\in\mathbb{C}G$, $\|f\|_*\leq 4K\|f\|_{s,(2,1)}$ which proves the statement.
\end{proof}
Lemma \ref{lem:operateur quotient vs hybride} has a few consequences for the rapid decay property for a pair of groups.
\begin{cor}
\label{cor:H moyenable}Let $G$ be a group and $H$ a subgroup.
If $\lambda_{G/H}\prec \lambda_G$, then $\|f\|_*=\|f\|_{h}$ for all $f\in\mathbb{R}_+G$. If moreover $G$ has the rapid decay property, then so does the pair $(G,H)$ and in particular the pair $(G,H)$ has rapid decay for any amenable subgroup $H$.
\end{cor}
\begin{proof}
 The weak containment $\lambda_{G/H}\prec \lambda_G$, combined with Lemma \ref{lem:operateur quotient vs hybride} (1), implies that $\|f\|_*=\|f\|_{h}$. Hence, if $G$ has rapid decay property, then so does the pair $(G,H)$ since there is $s\geq 0$ such that for any $f\in\cg$
\[
\|f\|_{h}=||f||_*\leq C\|f(1+\ell)^s\|_{2}\leq C\|f(1+\ell)^s\|_{(2,1)}.
\]
In particular, when $H$ is an amenable subgroup of a group $G$ with the rapid decay property, then the pair $(G,H)$ has rapid decay property as well since the trivial representation of $H$ is weakly contained in the regular one, which implies by induction that $\lambda_{G/H}\prec \lambda_G$ (see the proof of Corollary G.3.8 and Theorem F.3.5 in \cite{bekka2008kazhdan}).
\end{proof}
The following proposition is a generalization to non normal subgroup of the fact that the rapid decay property is stable by extension by a subgroup of polynomial growth for the induced length (see \cite{joli}).
\begin{prop}
\label{prop:H poly}
Let $G$ be a finitely generated group, the following are equivalent.
\begin{enumerate}
    \item The group $G$ has the rapid decay property.
    \item For any subgroup $H<G$ of polynomial growth for the length function induced by $G$, the pair $(G,H)$ has the rapid decay property.
     \item There exists a subgroup $H<G$ of polynomial growth for the length function induced by $G$ such that the pair $(G,H)$ has the rapid decay property.
\end{enumerate}
\end{prop}
\begin{proof}
$(1)\Rightarrow(2)$ If $H$ has polynomial growth, then it is amenable and $\lambda_{G/H}\prec \lambda_G$ so by Corollary \ref{cor:H moyenable} the rapid decay property implies it for the pair $(G,H)$. 
$(2)\Rightarrow(3)$ is clear so it remains to prove $(3)\Rightarrow(1)$. Since $H$ is of polynomial growth, according to Remark \ref{rmk : pol growth}, for any $f\in \mathbb{R}_+G$ with support in $B(R)$
\[
\|f\|_*\leq \|f\|_{h}\leq Q(R)\|f\|_{(2,1)}\leq 
Q(R)P(R)\|f\|_{2}
\]
where $Q$ is the polynomial from Proposition \ref{prop:equivalence rd_H} and $P$ the square root of the one from the growth of $H$.
\end{proof}
Recall that if $H$ is a subgroup of a finitely generated group $G$, the \textit{coefficients of the left quasi-regular representation} are given by 
 $$\{\langle\lambda_{G/H}(\gamma)\xi, \eta\rangle\}_{\gamma\in G}$$
 where $\xi,\eta\in \ell^2(G/H)$ are of norm one and $\langle\ ,\ \rangle$ is the scalar product on $\ell^2(G/H)$.
\begin{lem}
\label{lem:rdécroissance coefficients}
Let $G$ be a discrete group and $H$ a subgroup of $G$. The following equivalent properties imply both the rapid decay property for $G$ and for the pair $(G,H)$.
\begin{enumerate}
\item There exists  constants $M,s>0$ such that for all $\xi,\eta\in\ell^2(G/H)$ of norm one, we have
\[
\sum_{\gamma\in G} \frac{|\langle \lambda_{G/H}(\gamma)\xi,\eta\rangle|^2}{(1+\ell(\gamma))^{2s}}
\leq M.
\]
\item
There exists positive constants $K$ and $D$  such that for all
 $\xi,\eta\in\ell^2(G/H)$ of norm one, we have 
\[
\sum_{\gamma\in B(R)} |\langle \lambda_{G/H}(\gamma)\xi,\eta\rangle|^2\leq K(R+1)^D.
\]
\end{enumerate}
\end{lem} 
\begin{proof}
Notice first that $(1)$ and $(2)$ are equivalent. Indeed, suppose $(1)$, then for all  $R>0$
we have
\[
\frac{1}{(1+R)^s}
\sum_{\gamma\in B(R)} |\langle \lambda_{G/H}(\gamma)\xi,\eta\rangle|^2\leq {\sum_{\gamma\in B(R)} \frac{|\langle \lambda_{G/H}(\gamma)\xi,\eta\rangle|^2}{(1+\ell(\gamma))^{2s}}}
\leq M,
\]
which implies (2). Conversely, noticing that
\[
\sum_{\gamma\in G} \frac{|\langle \lambda_{G/H}(\gamma)\xi,\eta\rangle|^2}{(1+\ell(\gamma))^{2s}}=
\sum_{R\geq 0} \frac{1}{(1+R)^{2s}}\sum_{\gamma \in S(R)} |\langle \lambda_{G/H}(\gamma)\xi,\eta\rangle|^2,
\]
we deduce $(1)$ from (2) by taking $s>D+\frac{1}{2}$.

We now show that $(2)$ implies the rapid decay property both for $G$ and for the pair $(G,H)$. Let $f\in\mathbb{R}_+G$ with support in $B(R)$, $\xi\in\ell^2(G/H)$ of norm one and $\eta= \frac{\lambda_{G/H}(f)\xi}
 {\|\lambda_{G/H}(f)\xi\|_2}$, then
\[
\|\lambda_{G/H}(f)\xi\|_{2}=
\langle \lambda_{G/H}(f)\xi, \eta\rangle
=\sum_{\gamma\in B(R)} f(\gamma)\langle
 \lambda_{G/H}(\gamma)\xi,\eta\rangle\]
 \[
 \leq \|f\|_{2}\sqrt{\sum_{\gamma\in B(R)} |\langle \lambda_{G/H}(\gamma)\xi,\eta\rangle|^2}
 \leq K(R+1)^D\|f\|_{2},
\]
where the first inequality follows from Cauchy-Schwarz and the second one is from the assumption (2) with corresponding constants $K,D>0$. Now, according to Lemma \ref{lem:operateur quotient vs hybride} (2)
 \begin{eqnarray*}
\|f\|_*\leq\|f\|_{h}=\|\lambda_{G/H}(f)\|_{*}\leq K(R+1)^D\|f\|_{2}\leq K(R+1)^D\|f\|_{(2,1)},
\end{eqnarray*} 
which proves the rapid decay property both for $G$ and for the pair $(G,H)$. 
\end{proof}  
\begin{rmk}
When $H$ is the trivial subgroup, all properties of Lemma \ref{lem:rdécroissance coefficients} and Proposition \ref{prop:equivalence rd_H}
are equivalent and we recover the equivalent characterizations of the rapid decay property from Proposition 4.1 in \cite{IC}.
\end{rmk}
%%%%%%%%%%%%%%%%%%%%%%%%%%%%%%%%%%%%%%%%%%%
\section{Pairs of groups with and without relative rapid decay.}\label{relative}
%%%%%%%%%%%%%%%%%%%%%%%%%%%%%%%%%%%%%%%%%
In this section we investigate what pairs of groups $(G,H)$ have the relative rapid decay property.
%%%%%%%%%%%%%%%%%%%%%%%%%%%%%%%%%%%%%%
\subsection{Pairs of groups with a normal subgroup}
%%%%%%%%%%%%%%%%%%%%%%%%%%%%%%%%%%%%%
The goal of this part is to show that in case of a normal subgroup, the rapid decay property for a pair is equivalent to the rapid decay property for the quotient, showing that the rapid decay property for a pair can be thought of as the rapid decay property for a homogeneous space that is not necessarily a group. Theorem \ref{thm:cas normal} below implies that if the group $G$ sujects onto a group $Q$ with the rapid decay property via a homomorphism $\pi$, then the pair $(G,\ker(\pi))$ has the rapid decay property.
\begin{prop}\label{thm:cas normal} Let $G$ be a finitely generated group, and $H$ a normal subgroup of $G$, then the pair $(G,H)$ has the rapid decay property if and only if the quotient group $G/H$ has the rapid decay property.
\end{prop}
\begin{proof}With the notations or Remark \ref{rmk : pol growth}, if $S$ is a generating set of $G$, then for all $\gamma\in G$, we always have $\ell_Q(\pi(\gamma))\leq\ell_G(\gamma)$, where $\ell_Q$ denotes the word length function of $Q=G/H$ with respect to the generating set $\pi(S)$. Hence for $f\in\cg$ supported in $B(R)$, the support of $\pi_\sharp(f)$ is contained in $B_Q(R)$, the ball of radius $R$ in $Q$. Moreover, for all $f\in \mathbb{C}G$ and $\phi\in\etogh$ we have
\[
\|f*\phi||_{(2,1)}=||\pi_\sharp(|f*\phi|)||_2
\leq 
\|\pi_\sharp(|f|*|\phi|)\|_2=\|\pi_\sharp(|f|)*\pi_\sharp(|\phi|)\|_2.\]
 By assumption $Q$ has the rapid decay property, which implies that there exists constants $C,D\geq 0$ such that 
 \[\|\pi_\sharp(|f|)*\pi_\sharp(|\phi|)\|_2
\leq 
C(R+1)^D\|\pi_\sharp(|f|)\|_{2}\|\pi_\sharp(|\phi|)\|_{2}.\]
Hence, the pair $(G,H)$ has the rapid decay property. For the converse, we follow the proof of Theorem 3.1 of \cite{garncarek2015property}. We choose a section $s:Q\rightarrow G$ for $\pi$, satisfying that $\ell_Q=\ell_G\circ s$. In particular, $s$ sends the neutral element of $Q$ on the neutral element of $G$. Associated with the choice of $s$, we define a cocycle
$$c:Q\times Q \rightarrow H, (p,q)\mapsto s(p)s(q)s(pq)^{-1}$$
and an action
$$\alpha:Q\rightarrow {\rm Aut}(H), q\mapsto 
\alpha_q(h)=s(q)hs(q)^{-1}.$$
It follows that 
$(q,h)\rightarrow s(q)h$ identifies $G$ with $Q\times H$ with the
product structure $(p,h)(q,l)=(pq,h\alpha_p(l)c(p,q))$. With this choice of cross-section, we define $\widetilde{\ }:\mathbb{C}Q\to\cg$ by
\[\left\{
\begin{array}{l}
\tilde{f}(q,e)=f(q) ~~~~~~~~~~~~~~\hbox{ for all } q\in Q \\
\tilde{f}(q,h)=0  ~~~~~~~~~~~~~~~~~~~~~~\hbox{ for all }h\neq e,
\end{array}
\right.\]
where $f\in\mathbb{C}Q$ and $e$ is the neutral element in $H$. This is an algebra morphismes that isometrically extends to the $\ell^2$ completion, so that $\|\tilde{f}\|_{(2,1)}=\|f\|_{2}$, and $\widetilde{f*k}=\tilde{f}*\tilde{k}$ (the first convolution is in $Q$ and the second one in $G$). Let $f\in\mathbb{R}_+Q$ with support on $B_Q(R)$, then $\tilde{f}$ is supported on a ball of radius $R$ in $G$ and if $(G,H)$ has the rapid decay property, then there exists $C,D>0$ such that for all $\phi\in\ell^2(Q)_+$ then
$$
\|f*\phi\|_{2}=\|\tilde{f}*\tilde{\phi}\|_{(2,1)}\leq C(R+1)^D
\|\tilde{f}\|_{(2,1)}\|\tilde{\phi}\|_{(2,1)}=C(R+1)^D\|f\|_{2}\|\phi\|_{2}$$
and the group $Q$ has the rapid decay property.
\end{proof}
\begin{ex}\label{ex:normal} The semi-direct product $SL_2(\mathbb{Z})\ltimes\mathbb{Z}^2$ contains solvable subgroups with exponential growth, but the pair $(SL_2(\mathbb{Z})\ltimes\mathbb{Z}^2, \mathbb{Z}^2$) has rapid decay since the quotient is $SL_2(\mathbb{Z})$, which is a hyperbolic group, and those have the rapid decay property (see\cite{de1988groupes}). Baumslag-Solitar groups $BS(p,q)$ with $|p|\neq |q|$ have exponentially distorted amenable subgroups, and they admit a canonical surjective morphism over $\mathbb{Z}$. In particular for $n\in\mathbb{N}^*$, then $BS(1,n)\simeq\mathbb{Z}\ltimes\mathbb{Z}[\frac{1}{n}]$ and the pair $(BS(1,n), \mathbb{Z}[\frac{1}{n}])$
has the rapid decay property.
\end{ex}
We mention the following generalization of Proposition \ref{thm:cas normal}, which is the case where $H=K$ in what follows.
\begin{prop}[Common quotient stability] Let $(G,H)$  be a pair of groups with the rapid decay property and let $K<H$ be normal in $G$. Then the pair $(G/K, H/K)$ has rapid decay property as well.
\end{prop}
\begin{proof} Let $Q=G/K$ and $P=H/K$. Then $G$ (resp. $H$) is as a set the product $Q\times K$ (resp. $P\times K)$ with the group structure seen in the proof of Proposition \ref{thm:cas normal}.
Hence, for all $F\in \mathbb{R}_+Q$ (resp $\Psi\in l^{(2,1)}(Q,P)$) we can define 
the functions $f\in \mathbb{R}_+G$ and $\phi\in \etogh_+$ as follows:
for all $k\in K$, \[
\left\{
\begin{array}{l}
f(q,e)=F(q) ~~~~~~~~~~~~~~\forall q\in Q \\
f(q,h)=0  ~~~~~~~~~~~~~~~~~~~~~~\forall h\neq e

\end{array}
\right.
\]
\[
\left\{
\begin{array}{l}
\phi(q,e)=\Psi(q) ~~~~~~~~~~~~~~\forall q\in Q \\
\phi(q,h)=0  ~~~~~~~~~~~~~~~~~~~~~~\forall h\neq e

\end{array}
\right.
\]
Hence with this choice of functions, we have $\|f*\phi\|_{(2,1)(G,H)}=\|F*\Psi\|_{(2,1)(P,Q)}$. Since $(G,H)$ has rapid decay,
there exists $C,D>0$ such that 
$$\|f*\phi|_{(2,1)(G,H)}\leq C(1+R)^D\|f\|_{(2,1)(G,H)}\|\phi|_{(2,1)(G,H)}.$$ Hence,
$$\|F*\Psi\|_{(2,1)(P,Q)}\leq  C(1+R)^D \|F\|_{(2,1)(P,Q)}\|\Psi\|_{(2,1)(P,Q)}.$$
Which implies that $(P,Q):=(G/K,H/K)$ has the rapid decay property with the same degree as for the pair $(G,H)$.
\end{proof}
%%%%%%%%%%%%%%%%%%%%%%%%%%%%%%%%%%%%%%%%%
\subsection{Relative growth and co-amenability}\label{section: croissance et comoyenabilite}
%%%%%%%%%%%%%%%%%%%%%%%%%%%%%%%%%%%%%
This subsection aims to prove Theorem \ref{thm:cas coamenable} stated in the introduction. We start by explaining what it means for a quotient to have polynomial growth when the subgroup is not assumed normal.
\begin{definition}
Let $G$ be a finitely generated group with generating set $S$, and let $H<G$ be a subgroup. We denote by ${\bf S}(G,H,S)$ the {\it left Schreier graph with respect to the subgroup $H$}, which is the graph where vertices are the elements of $G/H$, and for all $s\in S$ two different cosets $gH$ and $sgH$ are linked by an edge of length 1.
\end{definition}
The graph ${\bf S}(G,H,S)$ is constructed as follows. Start with the base vertex, corresponding to the coset $H$, and use the left multiplication from $G$. If $s\in S\cap H$, then $sH=H$ and add a loop. Otherwise, $sH\not=H$ and there is an edge $\{H,sH\}$. Continue this process with the convention that when the left multiplication by an element of $S$ stabilizes the coset $gH$, then we add a loop on the vertex $gH$, and otherwise $gH$ and $sgH$ share an edge. Notice that he word length of an element $x\in G$ is always greater than the distance between $xH$ and $H$ in ${\bf S}(G,H,S)$. Different finite generating sets of $G$ give quasi-isometric Schreier graphs. When $H$ is a normal subgroup of $G$, then ${\bf S}(G,H,S)$ is quasi-isometric to a Cayley graph of the quotient group $G/H$. However, in general Schreier graphs are not simplicial as they have loops and multi-edges.
\begin{rmk}\label{rmk:transitive_action=schreier}
When $G=<S>$ acts transitively on a set $X$, the graph related to this action is the graph with vertex the elements of $X$, where $x$ and $y$ in $X$ are linked by an edges if $y=s\cdot x$,  where $s\in S$. This is the Schreier graph ${\bf S}(G,H,S)$ where $H$ is the stabilizer of a point (see Section 2.8 of  \cite{juschenko2019amenability}). Moreover, if $X$ is a metric space and $G$ acts by isometries, then the argument in the proof of Svarc-Milnor lemma (see IV.B Theorem 23 of \cite{dlh}) shows that the growth of $G/H$ (as in Definition \ref{def:growth}) is asymptotically bounded by the volume of the balls in $X$. The growth of Schreier graphs ${\bf S}(G,H,S)$ is well-defined and doesn't depend on generating sets.
\end{rmk}
\begin{lem}
\label{lem:polygrowth}
Let $G$ be a finitely generated group and $H$ a subgroup. If ${\bf S}(G,H,S)$ has polynomial growth bounded by a polynomial $Q$, then for every $f\in\cg$ supported on a ball of radius $R$ we have that
$$\|f\|_h\leq Q(R)\|f\|_{(2,1)}$$ 
and the pair $(G,H)$ has the rapid decay property.
\end{lem}
\begin{proof} For every $f\in\cg$ supported on a ball of radius $R$, then $\pi_\sharp(f)$ is supported on a ball of radius $R$ as well and
$$\|f\|_h\leq \|f\|_1=\|\pi_\sharp(|f|)\|_1\leq Q(R)\|\pi_\sharp(|f|)\|_2=Q(R)\|f\|_{(2,1)},$$
using Remark \ref{rmk : pol growth}.
\end{proof}
\begin{definition}\label{def:co-amenability}
Let $G$ be a group acting on a discrete set $X$, the action of $G$ on $X$ is called {\it amenable}\footnote{this is the amenability of the action in the sense of Eymard and not Zimmer, see \cite{AR} for the discussion.} if the following \textit{F\o lner condition} holds: for all $\epsilon>0$ and finite set $F\subseteq G$, there exists a non-empty finite set $V=V(\epsilon,F)$ in $X$ such that 
\[  
\frac{|F\cdot V\bigtriangleup V|}{|V|}\leq\epsilon.
\]
In particular, a subgroup $H$ of $G$ is said to be 
\textit{co-amenable} if the action of $G$ on $G/H$ by left translation is amenable.
\end{definition}
\begin{ex}
Finite index subgroups are co-amenable. More generally if $H$ is a subgroup of $G$ such that the growth of $G/H$ is subexponential, then $H$ is co-amenable in  $G$ (see \cite{introamenabl}). If $H$ is a normal subgroup of $G$, then the co-amenability of $H$ is equivalent to the amenability of the quotient. In \cite{diekert2017amenability}, it was shown that if $G$ is a certain kind of HNN-extension of $H$, then the spectral radius of the symmetric random walk on the Schreier graph ${\bf S}(G,H,S)$ is equal to 1, which is an equivalent definition of co-amenability.
\end{ex}
The following generalizes Leptin's criterion of amenability in \cite{leptin}.
\begin{prop}
\label{prop: égalité amenable l^1}
Let $G$ be a discrete group. A subgroup $H<G$ is co-amenable in $G$ if, and only if, for any $f\in\mathbb{R}_+G$, then $\|f\|_{h}=\|f\|_{1}$. 
\end{prop}
Notice that, just as the amenability characterization by the equality $\|f\|_*=\|f\|_1$ for every positive function doesn't imply that $\eog$ equals to $\cstr$ in general, the Banach algebras $\eog$ and $\cshyb$ differ.
\begin{proof}
According to Remark \ref{rmk:injection l^1} the inequality  $\|f\|_{h}\leq\|f\|_{1}$ always holds, so we start by proving the reverse inequality in the co-amenable case. By Definition \ref{def:co-amenability}, for any $\epsilon>0$ and $F=\hbox{supp}(f)$, there exists a finite set 
$V(F,\epsilon)=V\subseteq G/H$ such that  
$$\frac{|FV\bigtriangleup V|}{|V|}\leq\epsilon.$$
Consider $F^{-1}(FV)=\{\bar{g}\in FV| \ z^{-1}\bar{g}\in FV \hbox{ for all } z\in F \}$, then $V\subseteq F^{-1}(FV)$ and one checks that $FV= F^{-1}(FV)\cup (FV\bigtriangleup V)$, where
\[FV\bigtriangleup V=
\{\bar{a}\in FV|\bar{a} \notin V\}\cup\{\bar{a}\in V|\bar{a} \notin FV\},\]
so that F\o lner's condition implies that 
\[
 \frac{|F^{-1}(FV)|}{|FV|}\geq
 1-\epsilon.
\]
Now, take $\xi\in\ell^2(G/H)$ to be the normalized characteristic function of $FV$, so that $\|\xi\|_{2}=1$, and hence 
\begin{eqnarray*}
 \|\lambda_{G/H}(f)\xi\|_{2}^2 &=&
\frac{1}{|FV|}\sum_{\bar{g}\in G/H}\left(\sum_{\gamma\in G} f(\gamma)\mathbf{1}_{FV}(\gamma^{-1}\bar{g})\right)^2 \\
&\geq &\frac{1}{|FV|}\sum_{\bar{g}\in F^{-1}(FV)}\left(\sum_{\gamma\in G} f(\gamma)\mathbf{1}_{FV}(\gamma^{-1}\bar{g})\right)^2
=\|f\|_{1}^2(1-\epsilon),
\end{eqnarray*}
so that $\|\lambda_{G/H}(f)\|_*\geq\|f\|_1$ and we conclude using Lemma \ref{lem:operateur quotient vs hybride}. 

Conversely, if $\|f\|_{h}=\|f\|_{1}$ for any $f\in\mathbb{R}_+G$ and $S$ is a finite symmetric generating set of $G$, then for $f=\frac{\mathbf{1}_S}{|S|}$ we have that $\|f\|_{h}=1$ and hence, for every $\varepsilon>0$, there is a unit vector $\xi\in\ell^2(G/H)_+$ such that 
\[
\frac{1}{|S|}\langle\sum_{s\in S}\lambda_{G/H}(s)\xi,\xi\rangle>1-\varepsilon.
\]
Since $S$ is symmetric, for all $s\in S$ we have
\[
\|\lambda_{G/H}(s)(\xi)-\xi\|_{2}^2\leq 
\sum_{s\in S}\|\lambda_{G/H}(s)(\xi)-\xi\|_{2}^2
\leq 2|S|\varepsilon,
\]
and hence the left quasi-regular representation has almost invariant vectors, which is an equivalent definition of co-amenability (see Definition 2.2 and Remark 2.3 of \cite{stokke2006amenable}).
\end{proof}
We can now prove our main result of this section, namely that co-amenable pairs with the rapid decay property must have polynomial co-growth.
\begin{proof}[Proof of Theorem \ref{thm:cas coamenable}]
According to Lemma \ref{lem:polygrowth}, if the coset space $G/H$ has polynomial growth, then the pair $(G,H)$ has the rapid decay property so let us look at the converse, assuming that the pair $(G,H)$ has the rapid decay property and $H$ is co-amenable in $G$. For a $\bar{k}\in G/H$, let us enumerate the elements in the ball $B_{G/H}(\bar{k},R)=\{\bar{g}_1,\dots,\bar{g}_m\}$ where $m=|B_{G/H}(\bar{k},R)|$. In all the cosets $g_iH$ with $i=1,\dots,m$, we can chose one element $g_i\in\bar{g_i}$, of minimal length in $g_iH\subseteq G$ and  
 define the function $f\in\mathbb{R}_+G$ to be the characteristic function of this set $\{g_1,\dots,g_m\}$, so that $\|f\|_{1}=|B_{G/H}(\bar{k},R)|=m$, and $\|f\|_{(2,1)}=\sqrt{m}$. Moreover, $f$ is supported in a ball of radius $R+K$ in $G$ where $K=d(H,kH)$ in ${\bf S}(G,H,S)$. According to Proposition 
 \ref{prop: égalité amenable l^1} the rapid decay property for the pair $(G,H)$ implies that there is a polynomial $P$ such that
 \[|B_{G/H}(\bar{k},R)|=\|f\|_{1}=\|f\|_{h}\leq P(R+K)\|f\|_{(2,1)}=\sqrt{|B_{G/H}(\bar{k},R)|}.\]
Since this computation holds for all $\bar{k}\in G/H$ it proves poynomial co-growth of the subgroup $H$ in $G$. 
\end{proof}
\begin{ex}
\label{ex:RDH obstruction}
The Baumslag-Solitar group BS(1,2)$=\langle a, t| tat^{-1}=t^2\rangle$, is solvable and hence amenable, so that all its subgroup are co-amenable, and by Theorem~\ref{thm:cas coamenable} neither the pair $({\rm BS}(1,2),\left<a\right>)$ nor the pair $({\rm BS}(1,2),\left<t\right>)$ have the rapid decay property because in both cases the quotient has exponential growth.
\end{ex}
%%%%%%%%%%%%%%%%%%%%%%%%%%%%%%%%%%%
\subsection{Relatively hyperbolic groups}
\label{section relhyp}
%%%%%%%%%%%%%%%%%%%%%%%%%%%%%%%%%%
Corollary \ref{thm:RD_H relhyp} is a consequence of Proposition \ref{prop:H poly} combined with the fact that those relatively hyperbolic groups have rapid decay according to Theorem 0.1 of \cite{chruan}. The results of Jvbin Yao in \cite{JYao} show that in case $H<G$ is malnormal, then if the pair $(G,H)$ has rapid decay then that forces polynomial volume growth of $H$ and his argument gives the following example.
\begin{ex}\label{free}Let $G=F_3=\left<a,b,t\right>$ and $H=\left<a,b\right><G$ so that in particular $G$ is hyperbolic relative to $H$. Let $\xi_R$ denote the characteristic function of the ball of radius $R$ in $H$, and let $f_R=\xi_R*\delta_t\in {\mathbb R}G$. Then 
$$\|f_R\|_{(2,1)}=\|f_R\|_2=\sqrt{2\cdot 3^{R-1}}$$ 
since all the elements of the form $wt$ with $w\in H$ are in different $H$-cosets, but
$$\|f_R\|_h\geq\|f_R*\delta_{t^{-1}}\|_{(2,1)}=\|\xi_R\|_{(2,1)}=\|\xi_R\|_1=2\cdot 3^{R-1}$$
so that the pair cannot satisfy condition (3) in Proposition \ref{prop:equivalence rd_H}.
\end{ex}
%
%%%%%%%%%%%%%%%%%%%%%%%%%%%%%%%%%%%
\section{Consequences of the rapid decay property for a pair of groups}\label{consequences} 
%%%%%%%%%%%%%%%%%%%%%%%%
The rapid decay property for a pair of groups has consequences similar to the rapid decay property for a group, in terms of spectral embeddings, as well as in terms of random walks.
\subsection{Spectral embeddings}
%%%
Our goal is to prove Theorem \ref{prop:isoKtheory}, namely that the space $\textbf{H}_\ell^s(G,H)$ is a Banach algebra and that $K_i(\textbf{H}_\ell^s(G,H))\simeq K_i(\cshyb)$. We recall a few known facts on the K-theory of Banach algebras. Let $A$ be a unital Banach algebra, we denote by $A^\times$ the set of invertible elements of $A$. The {\it spectral radius} of an element $x\in A$ is given by
\[
\rho_A(x)=\sup\{|\lambda|\,|{\lambda\in \sigma_A(x)}\}=\lim_{n\to\infty} \|x^n\|_A^{\frac{1}{n}},
\]
where the {\it spectrum} $\sigma_A(x)=\{\lambda\in \mathbb{C}\ | \ x-\lambda1_A \notin A^\times \}$.
\begin{notation}\label{not:spec}
 For $G$ a finitely generated group, $H<G$ a subgroup and $x\in\cg$ we denote by $\rho_1(x)$ the spectral radius of $x$ as an element in the Banach algebra $\eog$, by $\rho_h(x)$ for the Banach algebra $\cshyb$, by $\rho_*(x)$ for the C*-algebra $\cstr$ and in case where the pair $(G,H)$ has the rapid decay property by $\rho_s(x)$ the spectral radius of $x$ as an element in the Banach algebra $\textbf{H}_\ell^s(G,H)$.
\end{notation}
\begin{rmk}\label{rmk:spec_s=t} If a pair of groups $(G,H)$ has the rapid decay property, then $\rho_s(f)=\rho_t(f)$ for any $f\in\cg$ and any $t\leq s$ large enough for $\textbf{H}_\ell^s(G,H)$ and $\textbf{H}_\ell^t(G,H)$ to be Banach algebras. Indeed, if $f$ is supported on $B(R)$, then the $n$-th convolution $f^{(n)}$ is supported on a ball of radius $nR$ and we have:
\[
\|f^{(n)}\|_{t,(2,1)}^{\frac{1}{n}}\leq \|f^{(n)}\|_{s,(2,1)}^{\frac{1}{n}}
\leq (1+nR)^{\frac{s}{n}} \|f^{(n)}\|_{t,(2,1)}^{\frac{1}{n}}
\]
so we conclude by taking the limit as $n$ goes to infinity.   
\end{rmk}
\begin{definition}
\label{def:morphism spec}
Let $A$ and $B$ two Banach algebras. An injective morphism $\phi: A\rightarrow B$ is
\begin{itemize}
\item \textit{Spectral} if for all $x\in A$, its image $\phi(x)\in B^\times$ if and only if $x\in A^\times$.
\item \textit{Relatively spectral} with respect to a dense sub-algebra $X\subseteq A$ if  $\phi:X\rightarrow B$ is spectral.
\item \textit{Isoradial} if $\rho_A(x)=\rho_B(\phi(x))$ for all $x\in A$.
\end{itemize}
\end{definition}
If the Banach algebras are completions of the group algebra $\mathbb{C}G$, it suffices to prove that the embedding is isoradial for all elements in $\mathbb{C}G$ to get an isomorphism in K-theory. Indeed, if $\phi$ has dense range and is isoradial, then $\phi$ is spectral and according to Nica in \cite{nica2008relatively}, a spectral morphism $\phi:A\rightarrow B$ between two Banach algebras induces an isomorphism in K-theory $\phi_*:K_*(A)\rightarrow K_*(B)$.
\begin{definition}
 Let $A(G)$ and $B(G)$ two Banach algebras which are completions of $\mathbb{C}G$ such that $A(G)\subseteq B(G)$. We say that $A(G)$ is {\it relatively spectral} in $B(G)$ if the inclusion is relatively spectral with respect to  $\mathbb{C}G$. 
\end{definition}
The Banach algebra $\eog$ is relatively spectral in $\cstr$ if and only if 
it is a  quasi-hermitian algebra (that is, self-adjoint elements in $\mathbb{C}G$ have real spectrum in $\eog$). This is always the case when the group $G$ has polynomial volume growth, but never happens when $G$ contains a 
free sub semi-group with two generators (see \cite{samei2020quasi} and \cite{jenkins1970symmetry}). We now turn to the proof of Theorem \ref{prop:isoKtheory}
\begin{proof}[Proof of Theorem \ref{prop:isoKtheory}]
We follow the proof of Propositions 8.10 and 8.16 in \cite{valette} to show that the inclusion ${\bf H}^{\infty}(G,H)\subseteq\cshyb$ is spectral. To do that we denote by $m_\ell$ the unbounded operator of pointwise multiplication by $\ell$ on $\etogh$ and define the following unbounded derivation on $\cshyb$, taking values in the bounded operators on $\etogh$:
\[
d:\psi\rightarrow [m_\ell,\psi].
\] 
We first show that
\begin{equation}\label{Der}
    D:=\bigcap_{k\geq 0}{\rm Dom}(d^k)={\bf H}^{\infty}_{\ell}(G,H),
\end{equation}
where ${\rm Dom}(d^k)$ is the domain of the derivation $d^k$. The inclusion $D\subseteq {\bf H}^{\infty}_{\ell}(G,H)$ is always true. Indeed, using an induction and the triangular inequality, we get that, for all $y\in G$ and $\phi,\psi\in\cg$
\[|(d^k(\psi)\phi)(y)|\leq \sum_{z\in G} |\psi(z)| |\phi(z^{-1}y)|\ell(z)^k=
|\psi|\ell^k*|\phi|(y)
\]
which implies that 
\begin{equation}\label{eq:deriv}
\|d^k(\psi)\phi\|_{(2,1)}\leq\|\,|\psi|\,\ell^k*|\phi|\,\|_{(2,1)}\leq 
\|\,|\psi|\,\ell^k\|_{h}\|\phi\|_{(2,1)}
\end{equation}
so that $\mathbb{C}G\subseteq D\subseteq\cshyb$. Moreover,  for all $k\geq 0$ and $\psi\in{\rm Dom}(d^k)$ we have that $(d^k(\psi)\delta_e)(\gamma)=\psi(\gamma)\ell(\gamma)^k\in \etogh$, which shows the first inclusion.

 If we furthermore assume that the pais $(G,H)$ has the rapid decay property, there exists constants $C,s>0$ such that for  $\psi\in {\bf H}^{\infty}_{\ell}(G,H)$ we have from the Equation (\ref{eq:deriv}) just above that
\[\|d^k(\psi)\|_{h}\leq\|\,|\psi|\ell^k\|_{h}\leq C\|\psi\|_{k+s,(2,1)}<\infty\]
which finishes to show Equation (\ref{Der}). According to Ji's result in \cite{Ji}, as stated in Proposition 8.12 of \cite{valette}, the algebra ${\bf H}^{\infty}_{\ell}(G,H)$ is spectral and dense in $\cshyb$. Since the inclusion
${\bf H}^{\infty}_{\ell}(G,H)\subseteq\cshyb$ factors through ${\bf H}^s_{\ell}(G,H)$, Proposition \ref{prop: Bs est une Banach alg} allows to conclude.
\end{proof}
\begin{lem}
\label{lem:spectral l^1}
For $G$ a finitely generated group and $H<G$ a subgroup such that the quotient $G/H$ has subexponential growth, then $\eog$ is a dense relatively spectral Banach sub-algebra of 
$\cshyb$. If moreover $G/H$ has polynomial growth, then for $s$ large enough the embedding of ${\bf H}^s_{\ell}(G,H)$ in $\eog$ is spectral.
\end{lem}
\begin{proof}
The density part is clear since those algebras are all completions of $\cg$. 
Let $f\in \mathbb{C}G$ with support in $B(R)$, by Remark  
\ref{rmk:injection l^1} we know that 
$\rho_{h}(f)\leq \rho_{1}(f)$ (see Notation \ref{not:spec}). Moreover $f^{(n)}$, the $n$-th convolution of $f$ with itself, is supported on ball of radius $nR$. The argument in Lemma \ref{lem:polygrowth} implies that
\[
\|f\|_{1}\leq \sqrt{\gamma_{G/H}(H,R)}
\|f\|_{(2,1)}\leq \sqrt{\gamma_{G/H}(H,R)}\|f\|_{h},
\]
so that, if $G/H$ has intermediate growth we get 
\[
\rho_{1}(f)=\lim_{n\to\infty}\|f^{(n)}\|_{1}^{\frac{1}{n}}\leq\lim_{n\to\infty}\left(\gamma_{G/H}(H,nR)\right)^{\frac{1}{2n}}\|f^{(n)}\|_{h}^{\frac{1}{n}}=
\rho_{h}(f),
\]
and $\eog$ is a dense relatively spectral Banach sub-algebra of 
$\cshyb$. In the particular case where $G/H$ has polynomial growth, then the pair $(G,H)$ has the rapid decay property according to Lemma \ref{lem:polygrowth} and for $s$ large enough ${\bf H}^s_{\ell}(G,H)\subseteq\eog$. Moreover, for all $n\in\mathbb{N}^*$ we have that
 \[
\|f^{(n)}\|_{s,(2,1)}\leq\|f^{(n)}\|_{h}\leq \|f^{(n)}\|_{1}\leq K \|f^{(n)}\|_{s,(2,1)}.
 \]
By taking $n$-th roots and a limit as $n$ goes to infinity, we deduce that $\rho_{s}(f)=\rho_{1}(f)=\rho_h(f)$.
\end{proof}
\begin{rmk}
The inclusion ${\bf H}_{\ell}^s(G,H)\subseteq C^*_r(G)$ of Lemma \ref{lem: Bs s'injecte dans C*} is in general not relatively spectral. When it is, then $\ell^1(H)$ is a relatively spectral sub-algebra of $C^*_r(H)$. The converse is false, as for instance in the solvable group $G=\mathbb{Z}\ltimes_\alpha \mathbb{Z}^2$ where $\alpha\in SL_2(\mathbb{Z})$ has an eigenvalue strictly greater than $1$, then $G$ contains a free semi-group of rank 2 (see \cite{chou}). Moreover, if $s$ and $t$ generate a free semi-group of rank 2, it was shown in \cite{samei2020quasi} that there exists 
$a_0,a_1, a_2\in\mathbb{C}$ such that for $x=a_0\delta_s+a_1\delta_{st}+a_2\delta_{st^2}$, then 
$\rho_{1}(x)=1$ whereas $\rho_*(x)<1$. Since $G/\mathbb{Z}^2=\mathbb{Z}$ is of polynomial growth, the pair $(G,{\mathbb{Z}^2})$ has the rapid decay property, and because of Lemma \ref{lem:spectral l^1}, we have
$\rho_{s}(x)=\rho_{1}(x)\neq \rho_*(x)$, showing that both spectra are different.
\end{rmk}
\begin{lem}
\label{lem: H sous-exp est sobolev spectral}
If $(G,H)$ is a pair of groups with the rapid decay property and $H$ is a subgroup of subexponential growth for the induced length, then for $s$ large enough, for all $f\in \mathbb{C}G$ we have 
\[
\lim_{n\to\infty} \|f^{(n)}\|^{\frac{1}{n}}_{s,2}=\lim_{n\to\infty} \|f^{(n)}(1+\ell)^s\|^{\frac{1}{n}}_{2}= \rho_{s}(f).
\]
\end{lem}
\begin{proof}
Let $f\in\mathbb{C}G$ with support in $B(R)$, denote by $B_H(R)$ the intersection of $B(R)$ with $H$. Using Cauchy-Schwarz inequality and Remark \ref{rmk:injection l^1} we get the following:
\[
\|f^{(n)}\|_{s,(2,1)}^{\frac{1}{n}}\leq \sqrt{|B_H(nR)|}^{\frac{1}{n}}\|f^{(n)}\|^{\frac{1}{n}}_{s,2}\leq \sqrt{|B_H(nR)|}^{\frac{1}{n}}\|f^{(n)}\|_{s,(2,1)}^{\frac{1}{n}}.
\]
According to Proposition \ref{prop: Bs est une Banach alg} for $s$ large enough ${\bf H}^s_{\ell}(G,H)$ is a Banach algebra, so we conclude by taking the limit as $n$ tends to infinity.
\end{proof}
\begin{cor}
\label{cor: cas spectral dans C*}
If $(G,H)$ is a pair of groups with the rapid decay property and $H$ is a subgroup of subexponential growth for the induced length from $G$, then ${\bf H}_{\ell}^s(G,H)$ is a relatively spectral sub-algebra of $C^*_r(G)$.
\end{cor}
\begin{proof}
According to Lemma \ref{lem: Bs s'injecte dans C*} there is $s$ large enough such that, for all $f\in\mathbb{C}G$, we have $\rho_*(f)\leq \rho_{s}(f)$. On the other hand, by  Hölder inequality, for $t\leq s$ we always have
 \[
\|f(1+\ell)^t\|_{2}\leq \|f(1+\ell)^s\|_{2}^{\frac{t}{s}}\|f\|_{2}^{1-\frac{t}{s}}
\]
which implies that for all $n\in\mathbb{N}^*$
\begin{equation}\label{eq:Holder}
 \left(\|f^{(n)}(1+\ell)^s\|_2^{\frac{-t}{s-t}}\|f^{(n)}(1+\ell)^t\|_2^{\frac{s}{s-t}}\right)^{\frac{1}{n}}
\leq\|f^{(n)}\|_{2}^{\frac{1}{n}}\leq\|f^{(n)}\|_*^{\frac{1}{n}}.   
\end{equation}
Since $H$ has subexponential growth in $G$, by Lemma \ref{lem: H sous-exp est sobolev spectral} and combining with Remark \ref{rmk:spec_s=t}, we have 
\[
\lim_{n\to\infty} \|f^{(n)}(1+\ell)^s\|^{\frac{1}{n}}_{2}=\rho_{s}(f)=\rho_{t}(f)=  
\lim_{n\to\infty} \|f^{(n)}(1+\ell)^t\|^{\frac{1}{n}}_{2}
\] 
and we conclude by taking the limit in Equation (\ref{eq:Holder}).
\end{proof}
%\begin{rmk}\label{Grigo} Because of Proposition \ref{prop:H poly}, the above Corollary \ref{cor: cas spectral dans C*} differs from the rapid decay property in case $H$ has intermediate growth in $G$, as the following example shows. Let $G=\mathbb{Z}*\mathcal{G}$, where $\mathcal{G}$ is a group of intermediate growth containing an element of infinite order (for instance Grigorchuk group \cite{grigorchuk}), so that $G$ contains a free subgroup of rank 2, which implies that $\eog$ is not a relatively spectral sub-algebra of $C^*_r(G)$ (see \cite{samei2020quasi}). Moreover, since $\mathcal{G}$ has intermediate growth, $G$ doesn't have the rapid decay property. But $G$ is hyperbolic relative to $\mathcal{G}$, hence with an analogue of Corollary \ref{thm:RD_H relhyp}, the pair $(G,\mathcal{G})$ has a subexponential version of the rapid decay property, and therefore Corollary \ref{cor: cas spectral dans C*} implies that $B^*_r(G,H)$ is a relatively spectral sub-algebra of $\cstr$ and hence $K_n(C^*_r(G))=K_n(B^*_r(G,H))$ but $\ell^1(G)$ is not relatively spectral in $\cstr$. \end{rmk}
%%%%%%%%%%%%%%%%%%%%%%%%%%
\subsection{Random walks and the probability of return to a subgroup}
Let $G$ be a group generated by a finite symmetric set $S$. A random walk on $G$ is given by a probability measure 
$\mu\in\mathbb{R}_+G$ with support in $S$, we call it {\it symmetric} if for all $\gamma\in G$, 
$\mu(\gamma^{-1})=\mu(\gamma)$. For such a random walk, the probability to go from $x$ to $y$ after $n$ steps in $G$ is given by 
$\mu^{(n)}(y^{-1}x)$  where $\mu^{(n)}$ denote the $n$-th convolution of $\mu$ with itself (see \cite{woess2000random}). More generally, if $G$ act transitively on a space $X$, for all $x,y\in X$ the \textit{transition probability} (from $x$ to $y$) after $n$ steps is given by 
\[P_n(x,y)=\sum_{z\in X} P_{n-1}(x,z)P(z,y)\]
where for every $x,y\in X$
\[
P(x,y)= \sum_{\{g\in G, g\cdot x=y\}} \mu(g).
\]
If a group $G$ acts transitively on a set $X$, a random walk $\mu$ on $G$ gives a random walk on $X$, and $X$ identifies with $G/H$, where $H$ is a point stabilizer of the $G$-action on $X$, see Remark \ref{rmk:transitive_action=schreier}.
\begin{definition}\label{def:spectralRad}
 Let $H$ be a subgroup of $G$ and ${\bf S}(G,H,S)$ be the associated Schreier graph. We denote by 
 $P_{2n}(H,H)$ the probability of returning on $H$ after $2n$ steps starting from $H$, and the {\it spectral radius of the random walk on ${\bf S}(G,H,S)$} is defined as
 \[
\rho_{\mu}^{G/H}=\limsup_{n\to\infty} P_{2n}(H,H)^{\frac{1}{2n}} .
\]
\end{definition}
\begin{rmk}
\label{rem: traduction random schreier cayley}
Since ${\bf S}(G,H,S)$ is a quotient of the Cayley graph 
of $G$, the probability $P_{2n}(H,H)$ can be viewed as the probability of ending at an element  $h\in H$ after $2n$ steps starting from another element $h'\in H$, namely
$$P_{2n}(H,H)=\sum_{h\in H} \mu^{(2n)}(h).$$
\end{rmk}
\begin{lem}
\label{lem:proba de retour sur H}
For $(G,H)$ a pair of groups and $\mu$ a symmetric random walk on $G$ then
$$P_{2n}(H,H)=\|\mu^{(n)}\|_{(2,1)}^2.$$
\end{lem}
\begin{proof}
Using Remark \ref{rem: traduction random schreier cayley}, we compute
\begin{eqnarray*}
P_{2n}(H,H)&=& \sum_{h\in H} \mu^{(2n)}(h)=\sum_{h\in H} \mu^{(n)}*\mu^{(n)}(h)=\sum_{h\in H}\sum_{z\in G}\mu^{(n)}(z^{-1})\mu^{(n)}(zh)\\
&=&\sum_{gH\in G/H}\sum_{z\in gH}\underbrace{\mu^{(n)}(z^{-1})}_{\mu^{(n)}(z)}\sum_{h\in H}\mu^{(n)}(zh)
\end{eqnarray*}
For any $z\in gH$, the last term of the above sum reads
\[
\sum_{h\in H}\mu^{(n)}(zh)=\sum_{y\in gH}\mu^{(n)}(y)=\|\mu^{(n)}|_{gH}\|_1
\]
Hence we conclude that
\[
P_{2n}(H,H)=\sum_{gH\in G/H}\|\mu^{(n)}|_{gH}\|_1^2=\|\mu^{(n)}\|_{(2,1)}^2.
\]
\end{proof}
\begin{prop}
\label{prop:RWspec=hyb}
Let $\mu$ be a symmetric random walk on a group $G$. Then, for any subgroup $H<G$,
$$\rho_\mu^{G/H}=\|\mu\|_{h}.$$
\end{prop}
\begin{proof}
That $\rho_{\mu}^{G/H}\leq\|\mu\|_{h}$ follows from Lemma \ref{lem:proba de retour sur H}, since for all $\in\mathbb{N}$,
\[
P_{2n}(H,H)=\|\mu^{(n)}\|_{(2,1)}^2\leq \|\mu\|_{h}^{2n}.
\] 
To show the reverse inequality, we adapt the proof of Lemma G.4.8 in \cite{bekka2008kazhdan}.
Since $\mu\in\mathbb{R}_+G$, 
by Lemma \ref{lem:operateur quotient vs hybride} we have that $\|\mu*\mu^*\|_{h}=\|\mu\|_{h}^2=\|\lambda_{G/H}(\mu)\|^2_{*}$. Let $\delta_{H}\in\ell^2(G/H)$ be the Dirac mass on the class $H$, and consider the operator $T=\lambda_{G/H}(\mu*\mu^*)=\lambda_{G/H}(\mu^2)$. A direct computation shows that 
\[
\langle T^n\delta_{H},\delta_{H} \rangle=\sum_{h\in H} \mu^{(2n)}(h)=P_{2n}(H,H).
\]
Using the spectral theorem for self-adjoint (positive) operator (see Chapter 6 of \cite{birman2012spectral}), there exists a Radon measure $\nu$ (which here is a probability measure) such that, for all $\epsilon>0$, and $n\in\mathbb{N}$ we have 
\[
P_{2n}(H,H)^{\frac{1}{2n}}=\left(\int_{\sigma(T)} x^n d\nu(x)\right)^{\frac{1}{2n}}\geq\left( \int_{\|\mu\|_{h}^2-\epsilon}^{\|\mu\|_{h}^2} x^n d\nu(x)\right)^{\frac{1}{2n}}.
\]
 By setting 
$K=\nu\left([\|\mu\|_{h}^2-\epsilon,\|\mu\|_{h}^2]\right)<\infty$, since $\|\mu\|_{h}\leq 1$ we get that 
\[
P_{2n}(H,H)^{\frac{1}{2n}}\geq K^{\frac{1}{2n}}\sqrt{\|\mu\|_{hyb}^2-\epsilon}> K^{\frac{1}{2n}}
(\|\mu\|_{h}-\epsilon).
\]
Taking the limit as $n\to +\infty$, we obtain that 
$\rho_{\mu}^{G/H} >\|\mu\|_{h}-\epsilon$ and conclude.
\end{proof}
\begin{proof}[Proof of Theorem \ref{thm:borne inf pour RW}] Since $\mu\in\mathbb{R}_+G$, by Lemma \ref{lem:operateur quotient vs hybride}, $\|\mu*\mu^*\|_{h}=\|\mu\|_{h}^2$ and hence $\|\mu\|_{h}^{2n}=\|\mu^{(n)}\|^2$. For any $n\in{\mathbb N}$, using Proposition \ref{prop:RWspec=hyb} and the rapid decay property for the pair $(G,H)$
$$\left(\rho_{\mu}^{G/H}\right)^{2n}=\|\mu^{(n)}\|^2_h\leq (n^d)^2\|\mu^{(n)}\|_{(2,1)}^2$$
and we conclude using Lemma \ref{lem:proba de retour sur H} and rearranging the terms.
\end{proof}
\begin{rmk}
According to Proposition \ref{prop: égalité amenable l^1}, the case $\rho=1$ corresponds to $H$ being a co-amenable subgroup of $G$.  When $H$ is amenable 
$\|\mu\|_*=\|\mu\|_{h}$ (see Corollary \ref{cor:H moyenable}) which implies that the 
the spectral radius of the random walk on the Schreier graph is the same as for the random walk on the Cayley graph. The converse is not true in general, see \cite{Ab_rt_2014} and \cite{anantharaman2003spectral} in which cases this equality implies the amenability of the subgroup $H$.
\end{rmk}
\section{Stability results and open questions}
\subsection{Stability questions}%
The following is a straightforward computation left to the reader.
\begin{prop}Let $(G,H)$ be a pair of groups with the rapid decay property, then for any finitely generated group $K$ the pair $(G\times K,H\times K)$ has the rapid decay property as well, with the same degree. If moreover $K$ has the rapid decay property, then the pair $(G\times K,H)$ has the rapid decay property as well.
\end{prop}
\begin{openquest} Given two pairs of groups with rapid decay $(G_1,H_1)$ and $(G_2,H_2)$, does the pair $(G_1\times G_2, H_1\times H_2)$ have rapid decay as well?
\end{openquest}
One can show that this is true if either one of the $H_i$'s or the quotients $G_i/H_i$ has polynomial growth, but the general case in unclear. 
\begin{openquest} Given two pairs of groups with rapid decay $(G_1,H_1)$ and $(G_2,H_2)$, as well as an extension $G_1\to E\to G_2$, what are the conditions under which this extension restricts to an extension $H_1\to F\to H_2$ such that the pair $(E,F)$ has rapid decay as well? In case we start with a central extension, can that condition be read on the cocycle that determines the extension?
\end{openquest}
%%%%
\subsection{K-theoretical questions}
Let $G$ be a finitely generated group and $H$ a subgroup of $G$. As seen in Lemma \ref{lem:spectral l^1}, the fact that 
$G/H$ as subexponential growth implies that the inclusion of $\eog$ in $\cshyb$ is relatively spectral with respect to $\cg$ and hence both algebras have the same $K$-theory.
Similarly, when $H$ has subexponential growth for the induced length, the spectral radius of a finitely supported function $f$ in $\cstr$ and $\cshyb$ is the same, and if the pair $(G,H)$ has the rapid decay property both algebras have the same $K$-theory (see Lemma \ref{lem: H sous-exp est sobolev spectral}). In fact it is known that for amenabe groups, we always have $K_i(\eog)\simeq K_i(\cstr)$ because amenable groups satisfy both the Bost and the Baum-Connes conjectures (see \cite{fusion}). The question that we ask is can we "cut" this consequence of amenability in two part in the following way:
\begin{openquest}
Let $H$ be a co-amenable (resp. amenable) subgroup of $G$,
   do $\eog$ (resp $\cstr$) and $\cshyb$ have the same K-theory?
\end{openquest}
Let us recall the following definition, from \cite{lafforgue2002k}
\begin{definition} A norm $\|\ \|_{A}$ on $\cg$ is called \textit{unconditional} if for all $f,\phi\in\cg$ such that $|f(g)|\leq |\phi(g)|$ for all $g\in G$, then $\|f\|_{A}\leq \|f\|_{A}$. In particular, for all 
 $f\in\cg$ we have $\|f\|_{A}=\|\,|f|\,\|_{A}$. A completion of $\cg$ with an unconditional Banach norm is called an {\it unconditional completion}. 
\end{definition}
The amenability of a group $G$ is equivalent to the inclusion $\eog\subseteq\cstr$ being an isometry on the positive valued functions. In this case we obtain the same K-theory, and ask the following
\begin{openquest}
    Let $A(G)$ and $B(G)$ two unconditional completions of $\cg$. Consider the algebras 
    $\mathcal{A}^*_r(G)$ and $\mathcal{B}^*_r(G)$ which are respectively the norm closures of $\lambda(\cg)$ in the bounded operators on $A(G)$ and $B(G)$, where $\lambda$ is the left regular representation.
    Suppose that for all $f\in\mathbb{R}_+G$ we have 
    $\|f\|_{\mathcal{A}^*_r(G)}=\|f\|_{\mathcal{B}^*_r(G)}$, 
    does it implies that  $\mathcal{A}^*_r(G)$ and $\mathcal{B}^*_r(G)$ have the same K-theory?
\end{openquest}
It would also be interesting to know when the K-theoric informations on the subgroup algebra $\ell^1(H)$ is enough to induce an isomorphism with the 
K-theory of $\cstr$. 
\begin{openquest}\label{question avec RD_H}
 Let $G$ be a group and $H<G$ be a subgroup. Suppose that the pair $(G,H)$ the rapid decay property, and that both algebras $\ell^1(H)$ and $C^*_r(H)$ have the same K-theory.
    Does it imply that $\cshyb$ and $\cstr$ have the same K-theory?
\end{openquest}
Notice that removing the rapid decay property assumption on the pair $(G,H)$ gives the following 
\begin{openquest}
\label{question sans RD_H}Let $G$ be a group and $H<G$ be a subgroup. Suppose that $\ell^1(H)$ and $C^*_r(H) $ have the same K-theory. Does it implies that $\cshyb$ and $\cstr$  have the same K-theory? 
\end{openquest}
The question is trivially true for $G=H$, and for $H=\{e\}$ is a reformulation of the combination of the Bost and Baum-Connes conjectures.
\bibliographystyle{alpha}
\bibliography{bibli}

\noindent
Laboratoire J.-A. Dieudonn\'e, Universit\'e C\^{o}te d'Azur, Parc Valrose, 06108 Nice Cedex 02,   France\\
\url{indira@unice.fr, zarka@unice.fr}
\end{document}